\documentclass[a4paper,12pt,reqno]{amsart}
\usepackage[a4paper,margin=1in]{geometry}

\usepackage{amsmath,amssymb,amsthm,mathtools,mathrsfs}
\usepackage{enumitem}
\usepackage{microtype}
\usepackage{xcolor}
\usepackage[
    colorlinks=true,
    linkcolor=blue!50!black,
    citecolor=green!40!black,
    urlcolor=blue!60!black
]{hyperref}
\usepackage[nameinlink,noabbrev]{cleveref}
\usepackage{tikz}
\usepackage{xurl}

\hypersetup{
    pdftitle={Length-generating functions for twisted involutions of Coxeter groups},
    pdfauthor={Ronald de Man}
}

\newtheorem{theorem}{Theorem}[section]
\newtheorem{proposition}[theorem]{Proposition}
\newtheorem{lemma}[theorem]{Lemma}
\newtheorem{corollary}[theorem]{Corollary}
\theoremstyle{definition}

\newtheorem{remark}[theorem]{Remark}
\newtheorem{example}[theorem]{Example}

\newcommand{\Inv}{\operatorname{\mathbf{I}}}
\newcommand{\DR}{\operatorname{D}_R}
\newcommand{\dpt}{\operatorname{dp}}
\newcommand{\cc}{\mathcal{C}}
\newcommand{\id}{\operatorname{id}}
\newcommand{\sgn}{\operatorname{sgn}}
\newcommand{\vrt}{\operatorname{Vert}}
\newcommand{\tr}{\operatorname{tr}}
\newcommand{\GL}{\operatorname{GL}}
\newcommand{\Q}{\mathbb Q}
\newcommand{\dm}{\diamond}

\title[Length-generating functions for twisted involutions]
{Length-generating functions for twisted involutions of Coxeter groups}

\author{Ronald de Man}
\email{ronalddeman@gmail.com}


\begin{document}

\begin{abstract}
For a Coxeter system $(W,S)$ of finite rank and an involutive automorphism
$\ast:W\to W$ which preserves $S$, we prove that the length-generating function
$\sum_{z\in\Inv_{W,\ast}}q^{\ell(z)}$ of the set of $\ast$-twisted involutions
is rational. We derive recurrence relations analogous to the well-known
recurrence relations expressing the growth series $W(q):=\sum_{w\in W}q^{\ell(w)}$
in terms of the growth series $W_I(q)$ of the parabolic subgroups of $W$.
Our result extends to any subset $\cc$ of $\ast$-twisted involutions that is
closed under $\ast$-twisted conjugation.
We use these recurrence relations to give an alternative proof of a
power-series identity for twisted involutions due to Lusztig.
\end{abstract}

\maketitle

\section{Introduction}\label{sec:intro}
Let $(W,S)$ be a Coxeter system of finite rank and let $\ast:W\to W$ be an
involutive automorphism preserving $S$. We are interested in counting
the lengths of the elements of the set of $\ast$-twisted involutions
$$
\Inv_{W,\ast}:=\{z\in W:z^\ast=z^{-1}\}
$$
and, more generally, of the subsets $\cc\subseteq\Inv_{W,\ast}$ closed
under $\ast$-twisted conjugation ($x\mapsto x^{-1}zx^{\ast}$).
Examples of such sets are the set of ordinary
involutions $\Inv_{W,\id}$ and the set of reflections
$T_W:=\{w^{-1}sw:w\in W,s\in S\}$.

For a subset $A\subseteq W$, we set $A(q):=\sum_{w\in A}q^{\ell(w)}$.
It is well-known that the length-generating function, or growth series,
$W(q):=\sum_{w\in W}q^{\ell(w)}$ is rational and satisfies recurrence relations
which express $W(q)$ in terms of the growth series $W_I(q)$ of the parabolic
subgroups $W_I$ of $W$.
Our main theorem gives similar recurrence relations for $\cc(q)$.
\begin{theorem}\label{thm:main}
Let $(W,S)$ be a Coxeter system of finite rank and let $\ast$ be an involutive
automorphism of W preserving $S$.
Let $\cc\subseteq\Inv_{W,\ast}$ be closed under twisted
conjugation and let $\cc_I=\cc\cap W_I$ for $I\subseteq S$.
Then $\cc(q)$ is rational. Moreover, if $W$ is infinite,
\begin{equation}\label{eq:main-infinite}
\boxed{\displaystyle
  \sum_{I\subseteq S, I=I^{\ast}}(-1)^{r_{\ast}(I)}
  \frac{\cc_I(q)}{W_I(q^2)}=0,}
\end{equation}
or, equivalently,
\begin{equation}\label{eq:main-infinite-proper}
  \boxed{\displaystyle
  \frac{\cc(q)}{W(q^2)}
  =\sum_{I\subsetneq S, I=I^{\ast}}
  (-1)^{r_{\ast}(S\setminus I)-1}\frac{\cc_I(q)}{W_I(q^2)}.}
\end{equation}
where $r_{\ast}(J)=|J/\langle \ast\rangle|$ denotes the number of orbits of $\ast$
acting on a $\ast$-stable subset $J\subseteq S$.

If $W$ is finite, let $w_0$ be the longest element and $N=\ell(w_0)$. Then
\begin{equation}\label{eq:main-finite}
  \boxed{\displaystyle
  \frac{\cc(q)+(-1)^{r_{\ast}(S)-1}q^{2N}\cc(-q^{-1})}{W(q^2)}
  = \sum_{I\subsetneq S,I=I^{\ast}}(-1)^{r_{\ast}(S\setminus I)-1}\frac{\cc_I(q)}{W_I(q^2)}.}
\end{equation}
\end{theorem}
By taking $\ast=\id$, $\cc=\Inv_{W,\id}$ and $\cc=T_W$, this proves the rationality
of the length-generating functions of the sets of involutions and reflections.

The rationality of $T_W(q)$ was listed in \cite{Brenti} as an open problem
originally posed by Stembridge but had been established earlier
by De Man \cite{DeMan99} by means of a modification of the Brink-Howlett finite
automaton (\cite{BH}) to count the vertices and edges of tree graphs
$G_\alpha$ associated with the positive roots of the Coxeter system.
More recently Biagioli et al. \cite{BHS} developed an alternative solution
in the case of affine Coxeter systems in the form of an explicit description of the
length-generating function of the set of reflections based on
a description of the Hasse diagram of the root poset.
Rationality in the affine case also follows from Evetts and Lathouwers
\cite{EveLat25} who, more generally, proved that the length-generating
functions of twisted conjugacy classes in virtually abelian groups are rational
and explicitly computable. The length polynomials of involutions in finite
Coxeter groups were studied and explicit formulas were given by
Hart and Rowley \cite{HarRow23}.

The remainder of the paper is organized as follows.
Section~\ref{sec:prelim} fixes notation and recalls various Coxeter-group
preliminaries.
In Section~\ref{sec:coxcomplex} we introduce facts about
Coxeter complexes and prove a contractibility statement.
Sections~\ref{sec:filtration} and \ref{sec:series} contain the main
topological ingredient of the proof.
For a twisted involution $z$, we introduce a filtration of the Coxeter complex
that is preserved by a simplicial involution associated with $z$. By
evaluating its Lefschetz numbers, we calculate the coefficients of a
power series associated with $z$. This approach was inspired by
Dymara's \cite{Dymara} derivation of the well-known recurrence relations
for the ordinary growth series of Coxeter groups.
In Section~\ref{sec:recurrence} we sum these power series over a subset $\cc$
closed under $\ast$-twisted conjugation to obtain Theorem~\ref{thm:main}.
For finite Coxeter systems, Section~\ref{sec:opposition} relates $\ast$-twisted
involutions to those for
the automorphism obtained by composing $\ast$ with opposition and records the
resulting reciprocity relations for length and twisted absolute length.
Section~\ref{sec:lusztig} combines these relations with the recurrences
from Sections~\ref{sec:series} and \ref{sec:recurrence} to
give an alternative proof of a power-series identity due to Lusztig
(\cite{Lusztig2012,Lusztig2015}).
Finally, Section~\ref{sec:classification} describes some aspects of the
practical use of the recurrence relations to compute the
length-generating function for a specific twisted conjugacy class of
twisted involutions.

\section{Preliminaries on Coxeter groups}\label{sec:prelim}
A \emph{Coxeter system} is a pair $(W,S)$ where $W$ is a group generated
by a set $S$ of \emph{simple reflections}, subject to the relation
$(st)^{m_{st}}=e$ for $s,t\in S$ with $m_{ss}=1$
and $m_{st}=m_{ts}\in\{2,3,\dots,\infty\}$ for $s\neq t$.
The matrix $(m_{st})_{s,t\in S}$ is called the \emph{Coxeter matrix} and
defines the \emph{Coxeter graph} $\Gamma(W)$ with the elements of $S$
as its vertices
and with an edge $\{s,t\}$, labelled $m_{st}$, whenever $m_{st}\neq 2$.

Throughout this paper $(W,S)$ will denote a Coxeter system of finite rank $|S|$.
Unless otherwise stated, proofs for the results mentioned in the remainder
of this section as well as further background can be found in
Humphreys \cite{Humphreys}.

The \emph{length} of an element $w\in W$ is defined as
$$
\ell(w):=\min\{n:w=s_1\cdots s_n, s_i\in S\},\qquad w\in W.
$$
For $I\subseteq S$, let $W_I:=\langle I\rangle\subseteq W$ be the
\emph{parabolic subgroup} generated by $I$.
The pair $(W_I,I)$ is again a Coxeter system with the same values $m_{st}$
for $s,t\in I$ and a length function that coincides with the restriction
of the length function of $W$.
We put
\begin{equation}\label{eq:cosetrep}
  W^I:=\{x\in W:\ell(xs)>\ell(x)\text{ for every }s\in I\}.
\end{equation}
Every left coset in $W/W_I$ can be written as $xW_I$
with $x\in W^I$ the unique minimal-length element of the coset.
Moreover, every $w\in W$ has a unique factorization or
\emph{parabolic decomposition}
\begin{equation*}
w=xu, \qquad x\in W^I,\quad u\in W_I,
\end{equation*}
for which
\begin{equation*}
\ell(xu)=\ell(x)+\ell(u).
\end{equation*}
By symmetry, the same applies to right cosets.

Every element $w\in {}^IW^J:={}^IW\cap W^J$ is the unique minimal
element of the double coset $W_IwW_J$ (\cite[Lemma~3.2.2]{DM20}).

Let $\Pi=\{\alpha_s:s\in S\}$, the set of \emph{simple roots} of $W$,
be the basis of an $\mathbb R$-vector
space $V$. We define a bilinear symmetric form on $V$ by
$(\alpha_s,\alpha_t)=-\cos(\pi/m_{st})$ (and $(\alpha_s,\alpha_t)=-1$
for $m_{st}=\infty$) and let $W$ act on $V$ by $s(v)=v-2(v,\alpha_s)\alpha_s$
for $s\in S$. The set $\Phi=\{w(\alpha_s):w\in W, s\in S\}$ is called the
\emph{root system} of $W$ and can be written as $\Phi^+\cup\Phi^-$ with
$\Phi^+=\{\sum_s{c_s\alpha_s}\in\Phi:c_s\ge 0\text{ for all }s\in S\}$
and $\Phi^-=-\Phi^+$.

An $S$-preserving automorphism $\delta:W\to W$ induces a linear
operator on $V$ defined by $\alpha_s\mapsto\alpha_{\delta(s)}$.
Since $\delta$ preserves the Coxeter matrix, it preserves the
bilinear form. It is easily verified that $\delta\circ s=\delta(s)\circ\delta$
for $s\in S$ and therefore $\delta\circ w=\delta(w)\circ\delta$
as linear operators on $V$ for all $w\in W$.

If the bilinear form is positive definite, the Coxeter group $W$ is
finite and referred to as \emph{spherical}. Otherwise, the group
is infinite. If $I\subseteq S$ and $W_I$ is spherical, we will say
that $I$ is spherical.
If $W$ is finite, it has a unique \emph{longest element} $w_0\in W$.
This element
satisfies $\ell(ww_0)=\ell(w_0w)=\ell(w_0)-\ell(w)$ for all $w\in W$.

If $\Gamma_1,\dots,\Gamma_m$ are the connected
components of $\Gamma(W)$, then $W\cong W_{I_1}\times\cdots\times W_{I_m}$,
where each $I_j$ is the set of vertices of $\Gamma_j$.
If $\Gamma(W)$ is connected, $W$ is \emph{irreducible}. The spherical
irreducible Coxeter groups have been classified into the types
$A_n$ $(n\ge 1)$, $B_n$ $(n\ge 2)$, $D_n$ $(n\ge 4)$, $E_6$, $E_7$,
$E_8$, $F_4$, $G_2$, $H_3$, $H_4$ and $I_2(m)$ $(m=5\text{ or }m\ge 7)$.
For convenience, we will identify a subset $I\subseteq S$
with the subgraph of $\Gamma(W)$ induced by $I$.

\section{Coxeter complexes}\label{sec:coxcomplex}

The \emph{Coxeter complex} associated with a Coxeter system is the simplicial
complex $\Sigma(W,S)$ whose simplices are the cosets
$xW_I$ with $x\in W$ and $I\subsetneq S$, where $xW_I$ is a face of $yW_J$
if and only if $yW_J\subseteq xW_I$. The maximal simplices $C_x:=xW_{\varnothing}$
are called \emph{chambers}. The vertices are the simplices $xW_{S\setminus \{s\}}$ with $s\in S$.
A simplex $xW_I$ is said to be of type $I$. It has dimension $|S|-|I|-1$, and
the chambers containing it are precisely $C_{xu}$ with $u\in W_I$.

For $A\subseteq W$, we let $\Sigma(A)$ denote the subcomplex consisting of all
simplices contained in at least one chamber $C_w$ with $w\in A$.
Equivalently,
\begin{equation}\label{eq:SigmaA}
\Sigma(A)=\{xW_I:xW_I\cap A\neq\varnothing\}.
\end{equation}
We write $|\Sigma(A)|$ for the geometric realization of $\Sigma(A)$.

For $w\in W$, let
$$
\DR(w):=\{s\in S:\ell(ws)<\ell(w)\}
$$
be its \emph{right descent set}. We say that $A\subseteq W$ is an
\emph{order ideal in the right weak order}
if for all $w\in W$ and $s\in\DR(w)$ we have $ws\in A$.

\begin{proposition}\label{prop:contractible}
Let $A\subseteq W$ be a finite and nonempty order ideal in the right weak
order and assume that $\DR(w)\neq S$ for all $w\in A$.
Then the geometric realization $|\Sigma(A)|$ is contractible.
\end{proposition}

\begin{proof}
Since $A$ is nonempty and for each $w\in W$ either $w=e$ or $\ell(ws)<\ell(w)$
for some $s\in S$, we see that $e\in A$.
Now choose an ordering
$$
e=w_1,w_2,\ldots,w_m
$$
of the elements of $A$ in nondecreasing order of length.
Let $K_j$ be the subcomplex consisting of the
chambers $C_{w_1},\ldots,C_{w_j}$ and all their faces. Thus
$K_m=\Sigma(A)$.

We prove by induction on $j$ that $|K_j|$ is contractible. This is clear for
$j=1$. Now let $j>1$ and put $w=w_j$. First we will show that
\begin{equation}\label{eq:chamber-intersection}
C_{w}\cap K_{j-1} =\bigcup_{s\in\DR(w_j)} wW_{\{s\}},
\end{equation}
as subcomplexes.

If $s\in\DR(w)$, then $\ell(ws)=\ell(w)-1$ and thus $ws\in A$.
Since we ordered $ws$ before $w$, $C_{ws}\subseteq K_{j-1}$.
Therefore $wW_{\{s\}}=C_w\cap C_{ws}\subseteq C_{w}\cap K_{j-1}$.

Conversely, suppose $xW_I$ is a simplex contained in $C_w\cap K_{j-1}$.
Since $C_w$ contains $xW_I$, we have $w\in xW_I$ and therefore $xW_I=wW_I$.
Likewise, the simplex $xW_I$ is contained in some $C_v$ with $v\in A$
ordered before $w$ and thus with $\ell(v)\le\ell(w)$ and $vW_I=xW_I=wW_I$.
Suppose that
$I\cap\DR(w)=\varnothing$. Then by \eqref{eq:cosetrep} we have $w\in W_I$,
which means that $w$ is the unique
minimal-length element of the coset $wW_I=vW_I$, contrary to $\ell(v)\le\ell(w)$.
Thus
$I\cap\DR(w)\neq\varnothing$ and we can choose $s\in I\cap\DR(w)$. Since
$W_{\{s\}}\subseteq W_I$, we have $wW_{\{s\}}\subseteq wW_I$ as cosets,
which means that the simplex $wW_I$ is a face of the panel $wW_{\{s\}}$.
This proves the reverse inclusion in \eqref{eq:chamber-intersection}.

By assumption $\DR(w)\neq S$, and since $w\neq e$, we also have
$\DR(w)\neq\varnothing$. Hence $C_w\cap K_{j-1}$ is a nonempty proper
union of the panels of $C_w$. Passing to geometric realizations, $|C_w|$ is
a simplex and
$$
|C_w\cap K_{j-1}| =\bigcup_{s\in\DR(w_j)} |wW_{\{s\}}|,
$$
is a nonempty proper union of its facets. Therefore $|C_w|$ strongly deformation
retracts onto $|C_w\cap K_{j-1}|$, fixing this subspace pointwise.
Extending this deformation
retraction by the identity on $K_{j-1}$ yields a deformation retraction of $|K_j|$
onto $|K_{j-1}|$. Hence, by induction $|K_j|$ is contractible.
\end{proof}

For a finite simplicial complex $X$ preserved by a simplicial automorphism $\tau$, 
we define the Lefschetz number
\begin{equation}\label{eq:lefschetzdef}
\Lambda(\tau,X):=\sum_{\substack{\sigma\in X\\\tau(\sigma)=\sigma}}
(-1)^{\dim \sigma}\sgn\left(\tau|_{\vrt(\sigma)}\right).
\end{equation}

\begin{lemma}\label{lem:lefschetz-fixed}
Let $X$ be a finite contractible simplicial complex preserved by a simplicial
automorphism $\tau$. Then
\begin{equation*}
  \Lambda(\tau,X)=1.
\end{equation*}
\end{lemma}

\begin{proof}
Choose an orientation for each simplex of $X$. For every $i\geq 0$, the
oriented $i$-simplices form a basis of the simplicial chain group
$C_i(X;\mathbb{Q})$. Since $\tau$ is a simplicial automorphism, it induces
a linear map
$$
  \tau_i \colon C_i(X;\mathbb{Q}) \longrightarrow C_i(X;\mathbb{Q}).
$$

Let $\sigma$ be an $i$-simplex of $X$. If $\tau(\sigma)\neq\sigma$,
then $\tau_i$ maps the basis vector corresponding to $\sigma$, up to sign,
to the basis vector corresponding to the distinct simplex $\tau(\sigma)$.
Thus $\sigma$ makes no contribution to $\tr(\tau_i)$.

Suppose now that $\tau(\sigma)=\sigma$. If $[v_0,\ldots,v_i]$
represents the chosen orientation of $\sigma$, then
$$
\tau_i[v_0,\ldots,v_i] = [\tau(v_0),\ldots,\tau(v_i)]
= \sgn\bigl(\tau|_{\vrt(\sigma)}\bigr) [v_0,\ldots,v_i].
$$
Hence
$$
\tr(\tau_i) = \sum_{\substack{\sigma\in X\\ \dim\sigma=i,\ \tau(\sigma)=\sigma}}
\sgn\bigl(\tau|_{\vrt(\sigma)}\bigr).
$$
By the definition of the Lefschetz number, it follows that
$$
\Lambda(\tau,X) = \sum_{i\geq 0}(-1)^i
\tr \bigl( \tau_i\colon C_i(X;\mathbb{Q})\to C_i(X;\mathbb{Q}) \bigr).
$$
Since $\tau$ is simplicial, the maps $\tau_i$ commute with the boundary
maps and therefore define a chain map. By the Hopf trace formula,
$$
\Lambda(\tau,X) = \sum_{i\geq 0}(-1)^i
\tr \bigl( \tau_{\ast}\colon H_i(X;\mathbb{Q})\to H_i(X;\mathbb{Q}) \bigr).
$$
Since $X$ is contractible,
$$
H_i(X;\mathbb{Q})=0 \quad\text{for } i>0,
\qquad H_0(X;\mathbb{Q})\cong\mathbb{Q}.
$$
Moreover, since $X$ is connected, $\tau$ acts trivially on $H_0(X;\mathbb{Q})$.
Therefore
$$
\Lambda(\tau,X)=1.
$$
\end{proof}

\section{A filtration for twisted involutions}\label{sec:filtration}
Let $\ast$ be an involutive automorphism of $W$ preserving $S$.
For a twisted involution $z\in\Inv_{W,\ast}$, we define
$$
\tau_z:\Sigma\to\Sigma,\quad \tau_z(xW_I):=zx^{\ast}W_{I^{\ast}}.
$$
Since $zz^{\ast}=e$, we have $\tau_z^2(xW_I)=xW_I$. Thus $\tau_z$ is
a simplicial involution of $\Sigma$.
For $x\in W^I$ we have
\begin{equation}\label{eq:tau-coset}
\tau_z(xW_I)=xW_I\quad\Longleftrightarrow\quad I=I^{\ast}
\text{ and } h:=x^{-1}zx^{\ast}\in W_I.
\end{equation}
Moreover, $hh^{\ast}=x^{-1}zx^{\ast}(x^{\ast})^{-1}z^{\ast}x=e$ and thus
$h\in\Inv_{W_I,\ast}\subseteq\Inv_{W,\ast}$.

We define a filtration of $W$ by the sets
\begin{equation*}
A_k(z):=\{w\in W:\ell(w)+\ell(zw^{\ast})\le k\}.
\end{equation*}

\begin{lemma}\label{lem:akz-filtration}
We have $A_k(z)=\varnothing$ for $k<\ell(z)$ and $A_k(z)\neq\varnothing$
for $k\ge\ell(z)$. Moreover, if $W$ is finite and $N=\ell(w_0)$, then
$w_0\notin A_k(z)$ for $k<2N-\ell(z)$ and $A_k(z)=W$ for $k\ge 2N-\ell(z)$.
\end{lemma}

\begin{proof}
For $w\in W$, we have $\ell(z)=\ell((zw^{\ast})(w^{\ast})^{-1})\le\ell(zw^{\ast})+\ell(w)$
with equality occurring for $w=e$. Thus
\begin{equation}\label{eq:F-min}
\min_{w\in W}\left(\ell(w)+\ell(zw^{\ast})\right)=\ell(z).
\end{equation}
Hence $A_k(z)=\varnothing$ for $k<\ell(z)$ and $A_k(z)\neq\varnothing$ for $k\ge\ell(z)$.

Now suppose that $W$ is finite.
Since $w\mapsto ww_0$ is a bijection commuting with $w\mapsto zw^{\ast}$, and since
$\ell(ww_0)=N-\ell(w)$ for all $w\in W$, \eqref{eq:F-min} translates into
\begin{equation*}
\max_{w\in W}\left(\ell(w)+\ell(zw^{\ast})\right)=2N-\ell(z)
\end{equation*}
with equality for $w=w_0$.
Hence $w_0\notin A_k(z)$ for $k<2N-\ell(z)$ and $A_k(z)=W$
for $k\ge 2N-\ell(z)$.
\end{proof}

\begin{lemma}\label{lem:akz-ideal}
For $k\ge 0$, the set $A_k(z)$ is a finite order ideal in the right weak order.
Moreover, $A_k(z)$ is stable under the map $w\mapsto zw^{\ast}$ and therefore $\tau_z$
preserves $\Sigma(A_k(z))$.
\end{lemma}

\begin{proof}
Finiteness follows from $\ell(w)\le k$ for all $w\in A_k(z)$.
Suppose that $w\in A_k(z)$ and $s\in S$ satisfy $\ell(ws)=\ell(w)-1$.
Then
$$
\ell(z(ws)^{\ast})=\ell((zw^{\ast})s^{\ast})\le\ell(zw^{\ast})+1.
$$
Consequently, $ws\in A_k(z)$.
Furthermore, if $w\in A_k(z)$, then
$$
\ell(zw^{\ast})+\ell(z(zw^{\ast})^{\ast})=\ell(zw^{\ast})+\ell(w)\le k.
$$
So $zw^{\ast}\in A_k(z)$. It follows that $\tau_z$ preserves $\Sigma(A_k(z))$.
\end{proof}

We set
$$
\Sigma^{\tau_z}:=\{\sigma\in\Sigma:\tau_z(\sigma)=\sigma\}.
$$
\begin{lemma}\label{lem:minF}
If $xW_I\in\Sigma^{\tau_z}$ with $x\in W^I$, then
\begin{equation}\label{eq:minF}
  \min_{u\in W_I}\left(\ell(xu)+\ell(z(xu)^{\ast})\right)=2\ell(x)+\ell(x^{-1}zx^{\ast}).
\end{equation}
Consequently, for $k\ge 0$,
\begin{equation}\label{eq:fixed-sublevel}
  \Sigma(A_k(z))\cap \Sigma^{\tau_z}
  =\{xW_I\in \Sigma^{\tau_z}:2\ell(x)+\ell(x^{-1}zx^{\ast})\le k\}.
\end{equation}
\end{lemma}

\begin{proof}
If $xW_I\in\Sigma_{\tau_z}$, then $I=I^{\ast}$ by \eqref{eq:tau-coset}.
If $x\in W^I$ and $u\in W_I$, then
$\ell(xu)=\ell(x)+\ell(u)$ by parabolic length additivity.
Set $h=x^{-1}zx^{\ast}$. Since $h\in W_I$ and thus $hu^{\ast}\in W_I$, we have
$\ell(z(xu)^{\ast})=\ell((zx^{\ast})u^{\ast})=\ell(xhu^{\ast})=\ell(x)+\ell(hu^{\ast})$.
Hence
$$
\ell(xu)+\ell(z(xu)^{\ast})=2\ell(x)+\ell(u)+\ell(hu^{\ast}).
$$
Applying the triangle inequality to $h=(hu^{\ast})(u^{\ast})^{-1}$ gives
$\ell(h)\le \ell(hu^{\ast}) + \ell(u)$, with equality occurring for $u = e$.
Thus
$$
  \min_{u\in W_I}\left(\ell(xu)+\ell(z(xu)^{\ast})\right)=2\ell(x)+\ell(h),
$$
which is \eqref{eq:minF}.

By \eqref{eq:SigmaA}, a simplex $xW_I$ belongs to $\Sigma(A_k(z))$ if and only
if $xu\in A_k(z)$ for some $u\in W_I$. If $xW_I\in\Sigma^{\tau_z}$, equation
\eqref{eq:minF} shows that this is equivalent to $2\ell(x)+\ell(x^{-1}zx^{\ast})\le k$,
proving \eqref{eq:fixed-sublevel}.
\end{proof}

We now set $X_k(z):=\Sigma(A_k(z))$. By Lemma~\ref{lem:akz-ideal}, $X_k(z)$ is
a finite simplicial complex.
\begin{lemma}\label{lem:Xkz}
If $W$ is infinite, then
$$
\Lambda(\tau_z,X_k(z))=\begin{cases}
0,&\text{if }k<\ell(z),\\
1,&\text{if }k\ge\ell(z).
\end{cases}
$$
If $W$ is finite and $N=\ell(w_0)$, then
$$
\Lambda(\tau_z,X_k(z))=\begin{cases}
0,&\text{if }k<\ell(z),\\
1,&\text{if }\ell(z)\le k<2N-\ell(z),\\
1+(-1)^{n-1+\ell(z)}\sgn({\ast}),&\text{if }k\ge 2N-\ell(z),
\end{cases}
$$
where $n=|S|$ and $\sgn({\ast})$ is the sign of the permutation $\ast:S\to S$.
\end{lemma}

\begin{proof}
By Lemma~\ref{lem:akz-filtration} we have
$\Lambda(\tau_z,X_k(z))=\Lambda(\tau_z,\varnothing)=0$
for $k<\ell(z)$.

Suppose that $k\ge\ell(z)$.
By Lemmas~\ref{lem:akz-filtration} and \ref{lem:akz-ideal},
$A_k(z)$ is a finite, nonempty ideal in the right weak order.
Since $\DR(w)=S$ only when $W$ is finite and $w=w_0$,
by Proposition~\ref{prop:contractible} and Lemma~\ref{lem:lefschetz-fixed}
$X_k(z)$ is contractible and $\Lambda(\tau_z,X_k(z))=1$
if either $W$ is infinite or $W$ is finite and $k<2N-\ell(z)$.

Now assume that $W$ is finite and $k\ge 2N-\ell(z)$. We have $A_k(z)=W$ and
$X_k(z)=\Sigma(W)$. We will write $\Sigma=\Sigma(W)$.
By the proof of Lemma~\ref{lem:lefschetz-fixed}, we have
$$
\Lambda(\tau_z,\Sigma) = \sum_{i\geq 0}(-1)^i
\tr \left((\tau_z)_{\ast}\colon H_i(\Sigma;\mathbb{Q})\to H_i(\Sigma;\mathbb{Q})
\right).
$$
If $W$ is finite, $|\Sigma|$ can be
realized as the unit sphere $S(V)\subseteq V$.
The simplicial map $\tau_z$ on $\Sigma(W)$ is induced on $S(V)$ by
the orthogonal automorphism $z\circ\ast:V\to V$, where the action
of $\ast$ is given by $\alpha_s\mapsto\alpha_{s^{\ast}}$.

We now assume that $n=\dim V=|S|\ge 2$.
Since $|\Sigma|=S(V)$ is an $(n-1)$-sphere, we have
$H_i(\Sigma;\Q)\cong\Q$ for $i=0$ and $i=n-1$, and
$H_i(\Sigma;\Q)=0$ for $0<i<n-1$. Moreover, $(\tau_z)_{\ast}|_{H_0}=\id$.
Since $H_{n-1}(\Sigma;\Q)\cong\Q$, the trace of
$\tau_z$ on $H_{n-1}(\Sigma;\Q)$ is the degree of the
self-map of $S(V)$ induced by $z\circ\ast$. The degree of
the map induced on a sphere by an invertible linear transformation is
the sign of its determinant. Since $z\circ\ast$ is an orthogonal automorphism,
$\det(z\circ\ast)\in\{\pm1\}$, and hence
$\tr\left((\tau_z)_{\ast}|_{H_{n-1}}\right)=\det(z\circ\ast)$.
We therefore obtain
$$
\Lambda(\tau_z,\Sigma)=1+(-1)^{n-1}\det(z\circ\ast).
$$
Since $\ast:V\to V$ permutes the simple roots, we have $\det(\ast)=\sgn(\ast)$.
Moreover, since $\det(s)=-1$ for $s\in S$, we have $\det(z)=(-1)^{\ell(z)}$.
This gives $\Lambda(\tau_z,\Sigma)=1+(-1)^{n-1+\ell(z)}\sgn(\ast)$.

For $n\le 1$, we have $\ast=\id$ and the value of $\Lambda(\tau_z,\Sigma)$
can be verified by direct evaluation of \eqref{eq:lefschetzdef}
for the three cases $(n,z)=(0,e),(1,e),(1,s)$.
\end{proof}

\section{A Lefschetz power series}\label{sec:series}
For a $\ast$-stable subset $J\subseteq S$, let
$$
r_{\ast}(J):=|J/\langle\ast\rangle|
$$
denote the number of orbits of $\ast$ acting on $J$. 
It is an elementary fact that
\begin{equation*}
\sgn\left(\ast|_J\right)=(-1)^{|J|-r_{\ast}(J)}.
\end{equation*}
We define the power series
\begin{equation}\label{eq:Et}
  E^{\ast}_z(q):=
  \sum_{xW_I\in\Sigma^{\tau_z}}
  (-1)^{r_{\ast}(S\setminus I)-1}q^{2\ell(x)+\ell(x^{-1}zx^{\ast})}.
\end{equation}
This is a well-defined formal power series because for $k\ge 0$ only
finitely many $x\in W$ have bounded length $\ell(x)\le k$.

By Lemma~\ref{lem:minF} and the definition \eqref{eq:lefschetzdef}
$$
\Lambda\left(\tau_z,X_k(z)\right) =
\sum_{\substack{\sigma=xW_I\in\Sigma^{\tau_z}\\
  2\ell(x)+\ell(x^{-1}zx^{\ast})\le k}}
  (-1)^{\dim\sigma}\sgn\left(\tau_z|_{\vrt(\sigma)}\right).
$$
A simplex $\sigma=xW_I\in\Sigma^{\tau_z}$ is of type $I$ with $I=I^{\ast}$.
Its vertices are indexed
by $S\setminus I$, and the permutation induced on its vertices is the permutation
induced by $\ast$ on $S\setminus I$. Hence
$\sgn\left(\tau_z|_{\vrt(\sigma)}\right)=(-1)^{|S\setminus I|-r_{\ast}(S\setminus I)}$.
Moreover, $\dim\sigma=|S\setminus I|-1$. This gives
$$
\Lambda\left(\tau_z,X_k(z)\right) =
\sum_{\substack{\sigma=xW_I\in\Sigma^{\tau_z}\\2\ell(x)+\ell(x^{-1}zx^{\ast})\le k}}
(-1)^{r_{\ast}(S\setminus I)-1}.
$$
On the other hand, the definition \eqref{eq:Et} of $E^{\ast}_z(q)$ gives
$$
[q^k]E^{\ast}_z(q) =
\sum_{\substack{\sigma=xW_I\in\Sigma^{\tau_z}\\ 2\ell(x)+\ell(x^{-1}zx^{\ast})=k}}
  (-1)^{r_{\ast}(S\setminus I)-1}
$$
as the coefficient of $q^k$. It follows that
\begin{equation}\label{eq:coefficient-jump}
  [q^k]E^{\ast}_z(q) =
  \Lambda\left(\tau_z,X_k(z)\right) - \Lambda\left(\tau_z,X_{k-1}(z)\right),
\end{equation}
where $X_{-1}(z)=\varnothing$.

\begin{theorem}\label{thm:fixpt}
Let $z\in\Inv_{W,\ast}$ be a twisted involution.

If $W$ is infinite, then
\begin{equation}\label{eq:fixpt-infinite}
  E^{\ast}_z(q)=q^{\ell(z)}.
\end{equation}

If $W$ is finite, let $w_0$ be the longest element of $W$ and let $N=\ell(w_0)$.
Then
\begin{equation}\label{eq:fixpt-finite}
  E^{\ast}_z(q)=q^{\ell(z)}+(-1)^{r_{\ast}(S)-1}(-q)^{2N-\ell(z)},
\end{equation}
\end{theorem}

\begin{proof}
This follows directly from \eqref{eq:coefficient-jump} and
Lemma~\ref{lem:Xkz} after observing that, in the finite case,
$\sgn(\ast)=(-1)^{|S|-r_{\ast}(S)}$ and thus
\begin{align*}
(-1)^{|S|-1+\ell(z)}\sgn(\ast)q^{2N-\ell(z)}&=(-1)^{r_{\ast}(S)-1+\ell(z)}q^{2N-\ell(z)}\\
&=(-1)^{r_{\ast}(S)-1}(-q)^{2N-\ell(z)}.
\end{align*}
Note that \eqref{eq:fixpt-finite} also holds if $\ell(z)=2N-\ell(z)$, i.e. if $z=w_0$.
\end{proof}

Let $F_{W,\ast}:=\{w\in W:w=w^{\ast}\}$ be the fixed-point subgroup of $W$ under $\ast$.
The following result appears in Steinberg \cite[Proposition~1.25]{Steinberg1968},
in the more general setting of arbitrary $S$-preserving automorphisms.
\begin{corollary}\label{cor:fwstar}
If $W$ is infinite, then
$$
\frac{1}{F_{W,\ast}(q)}=\sum_{I\subsetneq S,I=I^{\ast}}
  (-1)^{r_{\ast}(S\setminus I)-1}\frac{1}{F_{W_I,\ast}(q)}.
$$
If $W$ is finite, let $N$ be the length of its longest element. Then
$$
\frac{1+(-1)^{r_{\ast}(S)-1}q^N}{F_{W,\ast}(q)}=\sum_{I\subsetneq S,I=I^{\ast}}
  (-1)^{r_{\ast}(S\setminus I)-1}\frac{1}{F_{W_I,\ast}(q)}.
$$
\end{corollary}

\begin{proof}
We take $z=e$. Note that
$\Sigma^{\tau_e}=\{xW_I:I\subsetneq S,I=I^{\ast},x=x^{\ast}\}$.
Hence by definition \eqref{eq:Et}
\begin{equation}\label{eq:FW1}
\begin{aligned}
{E^{\ast}_e}(q)
  &=\sum_{\substack{I\subsetneq S\\I=I^{\ast}}}\sum_{\substack{x\in W^I\\x=x^{\ast}}}
    (-1)^{r_{\ast}(S\setminus I)-1}q^{2\ell(x)+\ell(x^{-1}x^{\ast})}\\
  &=\sum_{I\subsetneq S,I=I^{\ast}}(-1)^{r_{\ast}(S\setminus I)-1}
  \left(\sum_{x\in W^I,x=x^{\ast}}q^{2\ell(x)}\right).
\end{aligned}
\end{equation}
Let $I\subseteq S$ be $\ast$-stable. Then $w\in F_{W,\ast}$ factorizes
uniquely as $w=xu$ with $x\in W^I$, $u\in W_I$ but also as
$w=w^{\ast}=x^{\ast}u^{\ast}$ with $x^{\ast}\in W^{I^{\ast}}=W^I$
and $u^{\ast}\in W_{I^{\ast}}=W_I$, which implies that
$x=x^{\ast}$ and $u=u^{\ast}$. Thus
\begin{equation}\label{eq:FW2}
F_{W,\ast}(q)=\sum_{\substack{x\in W^I\\x=x^{\ast}}}
  \sum_{\substack{u\in W_I\\u=u^{\ast}}}q^{\ell(x)+\ell(u)}
=\left(\sum_{x\in W^I,x=x^{\ast}}q^{\ell(x)}\right)F_{W_I,\ast}(q).
\end{equation}
Applying Theorem~\ref{thm:fixpt} to evaluate $E^{\ast}_e(q)$, substituting
$q\mapsto q^{1/2}$ in \eqref{eq:FW1} and then combining with
\eqref{eq:FW2} gives the result.
\end{proof}

\begin{remark}
Taking $\ast=\id$, we obtain the well-known recurrence relation for $W(q)$:
\begin{equation*}
\sum_{I\subseteq S}\frac{(-1)^{|I|}}{W_I(q)}=
\begin{cases}
\frac{q^N}{W(q)},&\text{$W$ finite},\\
0,&\text{W infinite}.
\end{cases}
\end{equation*}
\end{remark}

\section{Recurrence relations for twisted involutions}\label{sec:recurrence}

In the introduction we defined $A(q):=\sum_{w\in A}q^{\ell(w)}$
for any subset $A\subseteq W$.
By the unique parabolic decomposition we have
$W(q)=W^I(q)W_I(q)$, where $W^I(q):=\sum_{x\in W^I}q^{\ell(x)}$.

\begin{proposition}\label{prop:sum-fixpt}
Let $(W,S)$ be a Coxeter system of finite rank and let $\ast:W\to W$ be an
involutive automorphism preserving $S$.
Let $\cc\subseteq\Inv_{W,\ast}$ be closed under $\ast$-twisted conjugation and
let $\cc_I=\cc\cap W_I$ for $I\subseteq S$. Then
\begin{equation}\label{eq:sum-fixpt}
  \sum_{z\in\cc}E^{\ast}_z(q)
  =W(q^2)
  \sum_{I\subsetneq S,I=I^{\ast}}
  (-1)^{r_{\ast}(S\setminus I)-1}\frac{\cc_I(q)}{W_I(q^2)}.
\end{equation}
\end{proposition}

\begin{proof}
Every simplex $\sigma$ has a unique expression $\sigma=xW_I$ with $I\subsetneq S$
and $x\in W^I$.
By \eqref{eq:tau-coset}, the set $\Sigma^{\tau_z}$ with $z\in\cc$ contains $xW_I$
if and only if $h=x^{-1}zx^{\ast}\in W_I\cap\cc=\cc_I$.
Using the definition \eqref{eq:Et} and rearranging terms, we get
\begin{align*}
  \sum_{z\in\cc}E^{\ast}_z(q)
  &=\sum_{z\in\cc}\sum_{xW_I\in\Sigma^{\tau_z}}
   (-1)^{r_{\ast}(S\setminus I)-1}q^{2\ell(x)+\ell(x^{-1}zx^{\ast})}\\
  &= \sum_{z\in\cc}\sum_{\substack{I\subsetneq S\\I=I^{\ast}}}
   \sum_{\substack{x\in W^I\\x^{-1}zx^{\ast}\in W_I}}
   (-1)^{r_{\ast}(S\setminus I)-1}q^{2\ell(x)+\ell(x^{-1}zx^{\ast})}\\
  &= \sum_{I\subsetneq S, I=I^{\ast}}\sum_{x\in W^I}\sum_{h\in\cc_I}
   (-1)^{r_{\ast}(S\setminus I)-1}q^{2\ell(x)+\ell(h)}\\
  &= \sum_{I\subsetneq S, I=I^{\ast}}(-1)^{r_{\ast}(S\setminus I)-1}W^I(q^2)\cc_I(q)\\
  &= W(q^2)\sum_{I\subsetneq S, I=I^{\ast}}(-1)^{r_{\ast}(S\setminus I)-1}\frac{\cc_I(q)}{W_I(q^2)}.
\end{align*}
\end{proof}

We can now prove the main theorem.
\begin{proof}[Proof of Theorem~\ref{thm:main}]
Assume first that $W$ is infinite. Summing equation \eqref{eq:fixpt-infinite}
of Theorem~\ref{thm:fixpt} over all $z\in\cc$ gives
$$
\sum_{z\in\cc}E^{\ast}_z(q)=\cc(q).
$$
Combining this with Proposition~\ref{prop:sum-fixpt} yields
$$
\frac{\cc(q)}{W(q^2)}
=\sum_{I\subsetneq S,I=I^{\ast}}(-1)^{r_{\ast}(S\setminus I)-1}\frac{\cc_I(q)}{W_I(q^2)}
$$
and rearranging gives \eqref{eq:main-infinite}.
Since each $\cc_I$ is closed under $\ast$-twisted conjugation within $W_I$, we can use
induction on $|S|$ to express $\cc(q)/W(q^2)$ as a rational combination
of functions $\cc_I(q)/W_I(q^2)$ with $W_I$ finite, which are all
quotients of polynomials. Since $W(q^2)$ is rational,
it follows that $\cc(q)$ is rational.

Now suppose that $W$ is finite. Then $\cc(q)$ is itself a polynomial and thus rational.
Equation \eqref{eq:fixpt-finite} gives
\begin{align*}
  \sum_{z\in\cc}E^{\ast}_z(q)
  &=\sum_{z\in\cc}
    \left(q^{\ell(z)}+(-1)^{r_{\ast}(S)-1}(-q)^{2N-\ell(z)}\right)\\
  &=\cc(q)+(-1)^{r_{\ast}(S)-1}q^{2N}\cc(-q^{-1}).
\end{align*}
Comparing this identity with \eqref{eq:sum-fixpt} gives \eqref{eq:main-finite}.
\end{proof}

\begin{corollary}\label{cor:invt}
The length-generating functions $\Inv_{W,\id}(q)$ and $T_W(q)$ counting the lengths of
involutions and reflections in $W$ are rational and satisfy the recurrence relations
\eqref{eq:main-infinite}-\eqref{eq:main-finite}.
\end{corollary}

\begin{proof}
Apply Theorem~\ref{thm:main} to $\cc=\Inv_{W,\id}$ and $\cc=T_W$
with $\ast=\id$ and
observe that $\Inv_{W,\id}\cap W_I=\Inv_{W_I,id}$ and $T_W\cap W_I=T_{W_I}$.
\end{proof}

\begin{remark}\label{rem:cancellation}
\leavevmode
\begin{enumerate}[label=(\roman*),ref=(\roman*)]
\item\label{it:first}
Since for finite W the degree of the polynomial $\cc(q)$ is at most $N$,
in \eqref{eq:main-finite} the individual terms of $\cc(q)=\sum_i a_iq^i$
and $$(-1)^{r_{\ast}(S)-1}q^{2N}\cc(-q^{-1})=\sum_i (-1)^{r_{\ast}(S)-1+i}a_iq^{2N-i}$$
can overlap only at $q^N$. If $r_{\ast}(S)\equiv N\mod 2$, they cancel each other out.
\item\label{it:coeff-an}
Since $w_0$ is a twisted involution, we know that $a_N=1$ if $\cc=\Inv_{W,\ast}$.
For $\ast=\id$ and $\cc=T_W$, it follows from $T_W(1)=|T_W|=N$ that
$a_N=N-(a_0+...+a_{N-1})$.
\item\label{it:class}
Let $\cc$ be a $\ast$-twisted conjugacy class of twisted involutions and
let $I\subseteq S$ be an inclusion-minimal
$\ast$-stable subset for which $\cc_I\neq\varnothing$ and thus $\cc_I(q)\neq 0$.
Then it follows from the recurrence
relation \eqref{eq:main-infinite-proper} applied to $(W_I,I)$ that $W_I$ cannot be
infinite and from the recurrence relation \eqref{eq:main-finite} applied to $(W_I, I)$
and \ref{it:first} that $\cc_I=\{w_{0,I}\}$ and $r_{\ast}(I)\equiv\ell(w_{0,I})\mod 2$.
\end{enumerate}
\end{remark}

We let $\Phi^+$ and $\Phi^-$ denote the sets of positive and negative roots of $(W,S)$
and define the \emph{depth} of a positive root by
$\dpt(\alpha):=\min\{\ell(w):w\in W, w(\alpha)\in\Phi^-\}$.
We set $\Phi_W(q):=\sum_{\alpha\in\Phi^+}q^{\dpt(\alpha)}$.
\begin{corollary}
The root-depth generating function $\Phi_W(q)$ is rational. If $W$ is infinite, it satisfies
$$
\frac{\Phi_W(q)}{W(q)}
=\sum_{I\subsetneq S}
  (-1)^{|S|-|I|-1}\frac{\Phi_{W_I}(q)}{W_I(q)}.
$$
If $W$ is finite, let $w_0$ be the longest element and $N=\ell(w_0)$. Then
$$
\frac{\Phi_W(q)+(-1)^{|S|}q^{N+1}\Phi_W(q^{-1})}{W(q)}
= \sum_{I\subsetneq S}(-1)^{|S|-|I|-1}\frac{\Phi_{W_I}(q)}{W_I(q)}.
$$
\end{corollary}

\begin{proof}
Each $\alpha\in\Phi^+$ corresponds bijectively to a reflection $t_\alpha\in T_W$
(see Björner and Brenti \cite[Proposition~4.4.5]{BB}) and
$\ell(t_\alpha)=2\dpt(\alpha)-1$ (see \cite[Lemma~1.2]{DeMan99}). Therefore
$\Phi_W(q^2)/q=T_W(q)$ and its rationality and the recurrence relations follow
from Corollary~\ref{cor:invt} and substituting $q\mapsto q^{1/2}$.
\end{proof}

Applying M\"obius inversion gives an analogue of Steinberg's formula in
\cite[Proposition~1.29]{Steinberg1968}, which expresses $W(q)$
in terms of the functions $W_I(q^{-1})$ associated with proper spherical
parabolic subgroups.
Let
$$
\mathcal S:=\{I\subseteq S:I=I^{\ast}, I\text{ is spherical}\}
$$
be the set of $\ast$-stable spherical subsets.

\begin{corollary}\label{cor:steinberg}
Let $(W,S)$ be a Coxeter system of finite rank. Let $\cc\subseteq\Inv_{W,\ast}$ be
closed under $\ast$-twisted conjugation and set $\cc_I=\cc\cap W_I$. Then
\begin{equation}\label{eq:spherical1}
  \frac{\cc(q)}{W(q^2)}
  =\sum_{I\in\mathcal S}
    (-1)^{r_{\ast}(I)}
    \frac{\cc_I((-q)^{-1})}{W_I((-q)^{-2})}.
\end{equation}
\end{corollary}

\begin{proof}
For each $\ast$-stable $J\subseteq S$, put
$$
R_J(q):=\frac{\cc_J(q)}{W_J(q^2)}.
$$
By Theorem~\ref{thm:main} applied to $(W_J,J)$,
\begin{equation}\label{eq:mobius-input}
\sum_{I\subseteq J,I=I^{\ast}}(-1)^{r_{\ast}(I)}R_I(q)=H_J(q),
\end{equation}
where, with $N_J=\ell(w_{0,J})$,
$$
H_J(q)=\begin{cases}
q^{2N_J}\cc_J(-q^{-1})/W_J(q^2),&J\in\mathcal S,\\
0,&J\notin\mathcal S.
\end{cases}
$$
Set $f(I)=(-1)^{r_{\ast}(I)}R_I$. Equation \eqref{eq:mobius-input} says
$H_J=\sum_{I\subseteq J,I=I^{\ast}}f(I)$. M\"obius inversion gives
$$
f(S)=\sum_{J\subseteq S,J=J^{\ast}}(-1)^{r_{\ast}(S\setminus J)}H_J
=\sum_{J\in\mathcal S}(-1)^{r_{\ast}(S\setminus J)}H_J.
$$
Multiplying by $(-1)^{r_{\ast}(S)}$ yields
$$
R_S=\sum_{J\in\mathcal S}(-1)^{r_{\ast}(J)}H_J.
$$
Now use the well-known fact that $W(q)=q^NW(1/q)$ for finite $W$
to obtain \eqref{eq:spherical1}.
\end{proof}

\begin{example}
Let $W=U_n$ be the universal Coxeter group of rank $n\ge2$, so $m_{st}=\infty$ for
$s\ne t$. Its only spherical standard parabolic subgroups are of rank 0 and 1.
For $I=\varnothing$,
$$
W_{\varnothing}=1,\qquad \Inv_{W_I,\id}(q)=1.
$$
For $I=\{s\}$,
$$
W_I(q)=A_1(q)=1+q,\qquad \Inv_{W_I,\id}(q)=1+q.
$$
Corollary~\ref{cor:steinberg} gives
$$
\frac{\Inv_{U_n,\id}}{U_n(q^2)}(q)=1-n\frac{1-q^{-1}}{1+q^{-2}}=1-n\frac{q^2-q}{q^2+1}.
$$
Since
$$
U_n(q)=\frac{1+q}{1-(n-1)q},
$$
we obtain
$$
\Inv_{U_n,\id}(q)=1+\frac{nq}{1-(n-1)q^2}.
$$
Thus the number of ordinary involutions of length $2k+1$ is $n(n-1)^k$ for $k\ge 0$.
\end{example}

\section{Twisted absolute length and opposition}\label{sec:opposition}

For twisted involutions, Hultman \cite{Hultman2005} introduced
the \emph{twisted absolute length} $\ell^{\ast}(z)$ of a $\ast$-twisted
involution $z$.
Here we will use Lusztig's \cite{Lusztig2012} equivalent eigenspace
characterization of $\ell^{\ast}$.
Let $V$ be the real reflection representation of $W$ and let
$\ast$ act on $V$ by $\alpha_s\mapsto\alpha_{s^{\ast}}$.
For $z\in\Inv_{W,\ast}$, the linear operator $z\circ\ast:V\to V$ is again an involution.
Letting $\nu(A)$ denote
the dimension of the $(-1)$-eigenspace of an involution $A\in\GL(V)$,
we set
$$
\ell^{\ast}(z):=\nu(z\circ\ast)-\nu(\ast).
$$
Observe that $\ell(z)\equiv \ell^{\ast}(z)\mod 2$, since
$$
(-1)^{\ell(z)}=\det(z)=\frac{\det(z\circ\ast)}{\det(\ast)}=(-1)^{\nu(z\circ\ast)-\nu(\ast)}
=(-1)^{\ell^{\ast}(z)}.
$$
Since $(w^{-1}zw^{\ast})\circ\ast=w^{-1}\circ(z\circ\ast)\circ w$ as linear
operators, $\ell^{\ast}(z)$ is constant on twisted conjugacy classes.
Moreover, $\ell^*$ is compatible with restriction to
parabolic subgroups.

For an ordinary involution $z\in\Inv_{W,\id}$, $\ell^{\id}(z)=\nu(z)$
coincides with the absolute length (or reflection length)
$$
\ell_T(z):=\min\{n:z=t_1\cdots t_n,t_i\in T_W\}.
$$
For finite Coxeter groups, this follows from Carter \cite{Car72};
for arbitrary Coxeter groups, it follows from Richardson's classification of
involutions \cite{Ric82} together with the finite result.

For the remainder of this section, we assume that $W$ is finite.
Conjugation by $w_0$ provides an $S$-preserving automorphism $\omega:W\to W$
defined by $\omega(w):=w_0ww_0$, which corresponds to the \emph{opposition
involution} on the Coxeter graph $\Gamma(W)$.
We set $\dm:=\omega\circ\ast$. Since $w_0^{\ast}=w_0$, $\omega$
and $\ast$ commute, and $\dm$ is again an $S$-preserving involutive
automorphism of $W$.
Moreover, since
$w_0(\alpha_s)=-\alpha_{w_0sw_0}=-\alpha_{\omega(s)}$ for $s\in S$,
we have $\omega=-w_0$ as linear operators on $V$.

\begin{proposition}\label{prop:stardiamond}
The map $z\mapsto zw_0$ gives a bijection $\Inv_{W,\ast}\to\Inv_{W,\dm}$ which
maps $\ast$-twisted conjugacy classes bijectively onto $\dm$-twisted
conjugacy classes. Moreover, $\ell(zw_0)=N-\ell(z)$ and
$\ell^{\dm}(zw_0)=a-\ell^{\ast}(z)$, where $N=\ell(w_0)$,
$a=\ell^{\ast}(w_0)=r_{\ast}(S)+r_{\dm}(S)-|S|$ and $a\equiv N\mod 2$.
\end{proposition}

\begin{proof}
If $z^{\ast}=z^{-1}$ then $(zw_0)^{\dm}=w_0(zw_0)^{\ast}w_0=w_0z^{-1}=(zw_0)^{-1}$.
Moreover, if $z'=x^{-1}zx^{\ast}$ is $\ast$-twisted conjugate to $z$, then
$z'w_0=x^{-1}zx^{\ast}w_0=x^{-1}zw_0(w_0x^{\ast}w_0)=x^{-1}(zw_0)x^{\dm}$ is
$\dm$-twisted conjugate to $zw_0$.

It is a standard fact that $\ell(zw_0)=N-\ell(z)$. Now let $r_{\ast}=r_{\ast}(S)$,
$r_{\dm}=r_{\dm}(S)$ and $n=|S|$.
Since $\ast$ maps $\alpha_s$ to $\alpha_{s^{\ast}}$ and each two-element $\ast$-orbit
contributes one $(-1)$-eigenvector, we have $\nu(\ast)=n-r_{\ast}$. Therefore
$\ell^{\ast}(w_0)=\nu(w_0\circ\ast)-\nu(\ast)=\nu(-\dm)-(n-r_{\ast})=r_{\dm}-(n-r_{\ast})
=r_{\ast}+r_{\dm}-n$.
We saw above that $a=\ell^{\ast}(w_0)\equiv\ell(w_0)=N\mod 2$.

Finally, observe that $\ell^{\ast}(z)=\nu(z\circ\ast)-\nu(\ast)=
\nu(z\circ\ast)-(n-r_{\ast})$ and that
$\ell^{\dm}(zw_0)=\nu(zw_0\circ\dm)-\nu(\dm)=\nu(-z\circ\ast)-\nu(\dm)=
n-\nu(z\circ\ast)-(n-r_{\dm})$.
Hence $\ell^{\dm}(zw_0)=a-\ell^{\ast}(z)$.
\end{proof}

The bivariate generating function
$$
B_{W,\ast}(q,t):=\sum_{z\in\Inv_{W,\ast}}q^{\ell(z)}t^{\ell^{\ast}(z)}
$$
records the lengths and twisted absolute lengths of twisted involutions.
\begin{corollary}\label{cor:bivar}
For $W$ finite, $N=\ell(w_0)$ and $a=r_{\ast}(S)+r_{\dm}(S)-|S|$, we have
\begin{equation*}
B_{W,\diamond}(q,t)=q^Nt^aB_{W,\ast}(q^{-1},t^{-1}).
\end{equation*}
\end{corollary}

\begin{proof}
This follows directly from Proposition~\ref{prop:stardiamond}.
\end{proof}

The relationship between the length-generating functions for the subsets of
fixed elements of $\ast$ and $\dm$ is as follows.
\begin{proposition}\label{prop:fw-opposition}
For $W$ finite and $a=r_{\ast}(S)+r_{\dm}(S)-|S|$, we have
\begin{equation*}
F_{W,\dm}(q)=\left(\frac{1+q}{1-q}\right)^a F_{W,\ast}(-q).
\end{equation*}
\end{proposition}

\begin{proof}
According to Steinberg \cite[Theorem~2.1]{Steinberg1968},
for a finite Coxeter group $W$ and an $S$-preserving automorphism
$\delta$ of $W$
$$
F_{W,\delta}(q)=\prod_{i=1}^n\frac{1-\epsilon_i q^{d_i}}{1-\epsilon_{0,i}q},
$$
where the $d_i$ are the degrees of homogeneous basic
invariants $f_i$ chosen so that $\delta(f_i)=\epsilon_if_i$, and where
the $\epsilon_{0,i}$ are the eigenvalues of $\delta$ acting on $V$.

Since $\ast$ is an involution, its eigenvalues
when acting on $V$ are $1$ with multiplicity $r_{\ast}(S)$ and $-1$ with
multiplicity $|S|-r_{\ast}(S)$. Likewise, the eigenvalues of $\dm$ when
acting on $V$ are $1$ with multiplicity $r_{\dm}(S)$ and $-1$ with
multiplicity $|S|-r_{\dm}(S)$. Furthermore,
since $\dm=-w_0\circ\ast$ and $w_0$ acts trivially on the
space of $W$-invariant polynomials, if $\ast(f_i)=\epsilon_if_i$, then
$\dm(f_i)=(-1)^{d_i}\epsilon_if_i$.
We therefore obtain
\begin{align*}
F_{W,\dm}(q)&=
\frac{\prod_{i=1}^n\left({1-\epsilon_i(-q)^{d_i}}\right)}
{(1-q)^{r_{\dm}(S)}(1+q)^{|S|-r_{\dm}(S)}}\\
&=\left(\frac{1+q}{1-q}\right)^a
\frac{\prod_{i=1}^n{1-\epsilon_i(-q)^{d_i}}}{(1+(-q))^{|S|-r_{\ast}(S)}(1-(-q))^{r_{\ast}(S)}}
=\left(\frac{1+q}{1-q}\right)^aF_{W,{\ast}}(-q).
\end{align*}
\end{proof}

\section{Lusztig's power-series identity}\label{sec:lusztig}

Lusztig \cite{Lusztig2012} introduced the formal power series
$$
L_{W,\ast}(q):=
\sum_{z\in\Inv_{W,\ast}}q^{\ell(z)}\left(\frac{q-1}{q+1}\right)^{\ell^{\ast}(z)}
$$
and proved that, if W is finite,
\begin{equation}\label{eq:lusztig}
L_{W,\ast}(q)=\frac{W(q^2)}{F_{W,\ast}(q)}.
\end{equation}
This result was subsequently extended to affine Coxeter groups of type
$\widetilde{A}_n$ by Marberg and White \cite{MarbergWhite2017} for $\ast=\id$
and to arbitrary Coxeter groups of finite rank by Lusztig \cite{Lusztig2015}.

In this section we will give an alternative proof of Lusztig's identity.
First, we observe that
$$
L_{W,\ast}(q)=\sum_k\left(\frac{q-1}{q+1}\right)^k\cc_k(q),
$$
where each $\cc_k=\{z\in\Inv_{W,\ast}:\ell^{\ast}(z)=k\}$ is closed under
$\ast$-twisted conjugation.

\begin{proposition}\label{prop:lusztig-infinite}
Suppose that identity \eqref{eq:lusztig} holds for all finite Coxeter systems $(W,S)$.
Then the identity also holds for all infinite Coxeter systems $(W,S)$ of finite rank.
\end{proposition}

\begin{proof}
Let $(W,S)$ be an infinite Coxeter system of finite rank.
We use induction on $|S|$ and may therefore assume that
$$
\frac{L_{W_I,\ast}(q)}{W_I(q^2)}=\frac{1}{F_{W_I,\ast}(q)}
$$
holds for all parabolic subgroups $W_I$ with $I\subsetneq S$ and $I=I^{\ast}$.
Hence, by Theorem~\ref{thm:main}, applied to each term $\cc_k(q)$,
and Corollary~\ref{cor:fwstar} we have
\begin{align*}
\frac{L_{W,\ast}(q)}{W(q^2)}
&=\sum_k\left(\frac{q-1}{q+1}\right)^k\frac{\cc_k(q)}{W(q^2)}
=\sum_{I\subsetneq S,I=I^{\ast}}(-1)^{r_{\ast}(S\setminus I)-1}
  \sum_k\left(\frac{q-1}{q+1}\right)^k\frac{\left(\cc_k\cap W_I\right)(q)}{W_I(q^2)}\\
&=\sum_{I\subsetneq S,I=I^{\ast}}(-1)^{r_{\ast}(S\setminus I)-1}
\frac{L_{W_I,\ast}(q)}{W_I(q^2)}
=\sum_{I\subsetneq S,I=I^{\ast}}\frac{(-1)^{r_{\ast}(S\setminus I)-1}}{F_{W_I,\ast}(q)}
=\frac{1}{F_{W,\ast}(q)},
\end{align*}
which is \eqref{eq:lusztig}.
\end{proof}

We now put $t:=t(q)=(q-1)/(q+1)$ and $a:=r_{\ast}(S)+r_{\dm}(S)-|S|$.
Interchanging $\ast$ and $\dm$, Proposition~\ref{prop:fw-opposition}
becomes
\begin{equation}\label{eq:fwrecip}
F_{W,\dm}(-q)=(-t)^aF_{W,\ast}(q).
\end{equation}

\begin{proposition}\label{prop:lusztig-finite}
Identity \eqref{eq:lusztig} holds for all finite Coxeter systems $(W,S)$.
\end{proposition}

\begin{proof}
We prove the identity simultaneously for all involutive automorphisms
by induction on $|S|$. For the trivial group $W=\{e\}$ we have
$L_{W,\ast}(q)=W(q^2)=F_{W,\ast}(q)=1$ and thus the identity holds for $|S|=0$.

Now assume that equation~\eqref{eq:lusztig} holds for all finite Coxeter systems
of rank smaller than $|S|$.
Let $a=r_{\ast}(S)+r_{\dm}(S)-|S|$ and $\epsilon_{\ast}=(-1)^{r_{\ast}(S)-1}$.
Applying the finite cases of Theorem~\ref{thm:main} and Corollary~\ref{cor:fwstar},
we obtain
$$
\frac{L_{W,\ast}(q)+\epsilon_{\ast}q^{2N}B_{W,{\ast}}(-q^{-1},t)}{W(q^2)}=
\frac{1+\epsilon_{\ast}q^N}{F_{W,\ast}(q)}.
$$
Equivalently, setting $H_{\ast}(q)=W(q^2)/F_{W,\ast}(q)$, we obtain
$$
L_{W,\ast}(q)+\epsilon_{\ast}q^{2N}B_{W,\ast}(-q^{-1},t)=(1+\epsilon_{\ast}q^N)H_{\ast}(q).
$$
Applying Corollary~\ref{cor:bivar} and using $t(-q)=1/t(q)$ yields
$$
B_{W,\ast}(-q^{-1},t)=(-q)^{-N}t^aB_{W,\dm}(-q,t(-q))=(-1)^Nq^{-N}t^aL_{W,\dm}(-q).
$$
We therefore obtain
\begin{equation}\label{eq:lineq1}
L_{W,\ast}(q)+\epsilon_{\ast}(-1)^Nq^Nt^aL_{W,\dm}(-q)=(1+\epsilon_{\ast}q^N)H_{\ast}(q),
\end{equation}
which is a first linear equation relating $L_{W,\ast}(q)$ and $L_{W,\dm}(-q)$.
To obtain a second linear equation in the same two unknowns,
we interchange $\ast$ and $\dm$. Making the substitution $q\mapsto -q$ and
using $t(-q)=1/t(q)$, we obtain
\begin{equation}\label{eq:lineq2}
\epsilon_{\dm}q^Nt^{-a}L_{W,\ast}(q)+L_{W,\dm}(-q)=(1+\epsilon_{\dm}(-1)^Nq^N)H_{\dm}(-q).
\end{equation}
The determinant of this system of two linear equations \eqref{eq:lineq1}
and \eqref{eq:lineq2} is
$$
1-\epsilon_{\ast}\epsilon_{\dm}(-1)^Nq^{2N},
$$
which is a nonzero polynomial in $q$. This means that the system has a unique
solution. We claim that it is $(L_{W,\ast}(q),L_{W,\dm}(-q))=(H_{\ast}(q),H_{\dm}(-q))$.
To verify \eqref{eq:lineq1}, we use \eqref{eq:fwrecip} to replace $H_{\dm}(-q)$
with $(-1)^{-a}t^{-a}H_{\ast}(q)$ and note that $a\equiv N\mod 2$:
\begin{align*}
L_{W,\ast}(q)+\epsilon_{\ast}(-1)^Nq^Nt^aL_{W,\dm}(-q)&=
H_{\ast}(q)+\epsilon_{\ast}(-1)^Nq^Nt^aH_{\dm}(-q)\\
&=H_{\ast}(q)+\epsilon_{\ast}q^NH_{\ast}(q)=(1+\epsilon_{\ast}q^N)H_{\ast}(q).
\end{align*}
To verify \eqref{eq:lineq2}, we use \eqref{eq:fwrecip} to replace $H_{\ast}(q)$
with $(-1)^at^aH_{\dm}(-q)$:
\begin{align*}
\epsilon_{\dm}q^Nt^{-a}L_{W,\ast}(q)+&L_{W,\dm}(-q)=\epsilon_{\dm}q^Nt^{-a}H_{\ast}(q)+H_{\dm}(-q)\\
&=\epsilon_{\dm}(-1)^aq^NH_{\dm}(-q)+H_{\dm}(-q)
=(1+\epsilon_{\dm}(-1)^Nq^N)H_{\dm}(-q).
\end{align*}
Hence $L_{W,\ast}(q)=H_{\ast}(q)=W(q^2)/F_{W,\ast}(q)$, which completes the induction step.
\end{proof}

\section{Twisted conjugacy classes}\label{sec:classification}

This section describes some aspects of the practical use of the
recurrence relations to compute the
length-generating function for a specific twisted conjugacy class
of twisted involutions.
As we observed in Remark~\hyperref[it:class]{\ref*{rem:cancellation}\ref*{it:class}},
each such twisted
conjugacy class $\cc$ of a Coxeter group $W$ contains an element which is
the longest element of a spherical parabolic subgroup $W_I$ with $I=I^{\ast}$.
A twisted conjugacy class $\cc$ can therefore be specified by means of a
$*$-stable spherical subset $I\subseteq S$. We note that this input enters
the recurrence computation through the possible cancellation of the leading
coefficient of $\cc_I(q)$ (see
Remark~\hyperref[it:first]{\ref*{rem:cancellation}\ref*{it:first}}).
To correct for these cancellations, we need criteria for determining for
which other spherical $*$-stable subset $J$ the longest element of $W_J$ is
$*$-twisted conjugate to the longest element of $W_I$.

Insofar as this section relates to the
classification of twisted involutions, the results are not new;
see Richardson \cite{Ric82} for the classification of ordinary involutions,
He \cite{He07} for extensions to the twisted case, and
Marquis \cite{Marquis25} for a full description of the (twisted) conjugacy
classes of an arbitrary Coxeter group.

We write $w_I$ for the longest element of a spherical parabolic subgroup $W_I$,
let $\omega_I:W_I\to W_I$ denote the opposition automorphism of $W_I$
defined by $w\mapsto w_Iww_I$, and set $\dm_I:=\omega_I\circ\ast|_{W_I}$.
Note that, in general, $\omega_J|_{W_I}\neq \omega_I$ for spherical subsets
$I\subsetneq J$. However, if $I$ is an irreducible component of $J$, then
$\omega_J|_{W_I}=\omega_I$ because the opposition involution on $J$ is
determined by the opposition
involution on each irreducible graph component of $J$.

Let $I$ and $J$ be two distinct $\ast$-stable subsets with $K:=I\cup J$ spherical.
Note that if $\omega_K(I)=J$, by uniqueness
of the longest element we have $\omega_K(w_I)=w_J$, which means that $w_I$ and
$w_J$ are $\ast$-twisted conjugate via the longest element $w_K$ of $K$.
If $J\setminus I$ consists of a single $\ast$-orbit, we call the transition
from $I$ to $\omega_K(I)=J$ an \emph{elementary move}.

We will use the following result due to Marquis \cite{Marquis25}.
It generalizes a result originally proved by Deodhar \cite{Deo82},
and in the finite case by Howlett \cite{How80}, to the twisted case.
\begin{lemma}\label{lem:marquis}
Let $I$ and $J$ be $\ast$-stable spherical subsets of $S$. Suppose there
exists $x=x^{\ast}\in W$ that has minimal length in $W_JxW_I$ and satisfies
$x(\Pi_I)=\Pi_J$. Then there exists a sequence $I=I_0,I_1,\dots,I_m=J$
of $*$-stable subsets $I_i$ of $S$ and elements $s_1,\dots,s_m\in S$
such that for $1\le i\le m$ we have
\begin{enumerate}[label=(\roman*),ref=(\roman*)]
\item
$I_i\setminus I_{i-1}=\{s_i,s_i^{\ast}\}$ (possibly with $s_i=s_i^{\ast}$),
\item
$K_i:=I_{i-1}\cup I_i$ is spherical, and
\item
$I_i=\omega_{K_i}(I_{i-1})$.
\end{enumerate}
\end{lemma}

\begin{proof}
This follows from \cite{Marquis25}, Proposition~4.11, after deleting any
repetitions from the sequence occurring there.
\end{proof}

We let $\ell^{\ast}(\cc)$ denote the twisted absolute length of an element of $\cc$.
Since $\ell^{\ast}$ is constant on $\ast$-twisted conjugacy classes, this is
well defined. By Proposition~\ref{prop:stardiamond} applied to $(W_I,I)$,
$\ell^{\ast}(w_I)=r_{\ast}(I)+r_{\dm_I}(I)-|I|$ for any $\ast$-stable
$I\subseteq S$.
We say that a $\ast$-stable $I\subseteq S$ is
\emph{inclusion-minimal for $\cc$} if $\cc\cap W_I\neq\varnothing$
and $\cc\cap W_{I'}=\varnothing$ for all $*$-stable $I'\subsetneq I$.

\begin{proposition}\label{prop:classif}
Let $(W,S)$ be a Coxeter system of finite rank and let $\cc$ be
a $\ast$-twisted conjugacy class of twisted involutions. Let $I$ denote a
$\ast$-stable subset of $S$.
\begin{enumerate}[label=(\roman*),ref=(\roman*)]
\item\label{it:class-minimal}
If $I$ is inclusion-minimal for $\cc$, then $I$ is spherical,
$\cc\cap W_I=\{w_I\}$ and $\ast|_{W_I}$ induces the opposition involution
on $I$ or, equivalently, $r_{\dm_I}(I)=|I|$. Conversely, if $I$ is
spherical and $r_{\dm_I}(I)=|I|$, then $I$ is inclusion-minimal for
the $*$-twisted conjugacy class of $w_I$.
\item\label{it:class-moves}
If $I$ is inclusion-minimal for $\cc$, then so is any subset of $S$ obtained
from $I$ by an elementary move. Moreover, if $J\subseteq S$ is another
$\ast$-stable inclusion-minimal subset for $\cc$, then $J$ can be
obtained from $I$ by a series of elementary moves.
\item\label{it:class-superset}
Let $I$ be inclusion-minimal for $\cc$ and let $J\subseteq S$ be $\ast$-stable
and spherical with $I\subseteq J$. Then $w_J$ is $\ast$-twisted conjugate
to $w_I$ if and only if $r_{\ast}(J)+r_{\dm_J}(J)-|J|=r_{\ast}(I)$.
\end{enumerate}
\end{proposition}

\begin{proof}
As observed in Remark~\hyperref[it:class]{\ref*{rem:cancellation}\ref*{it:class}},
if $I$ is an inclusion-minimal subset for which $\cc_I(q)\neq 0$, then
Theorem~\ref{thm:main} implies that $W_I$ is finite and $\cc_I(q)=q^{\ell(w_I)}$.
Thus $\cc\cap W_I=\{w_I\}$. Hence for $s\in I$ we have $s^{-1}w_Is^{\ast}=w_I$
and thus $\omega_I(s)=w_Isw_I=s^{\ast}$, i.e. $\omega_I$ and $\ast$ induce the same
permutation on $I$. This is equivalent to $\dm_I=\omega_I\circ\ast|_{W_I}$ being
the identity on $I$, i.e. $r_{\dm_I}(I)=|I|$.

Conversely, suppose $I$ is spherical and $r_{\dm_I}(I)=|I|$. Then
by Proposition~\ref{prop:stardiamond}
$\ell^{\ast}(w_I)=r_{\ast}(I)$. Let $J\subseteq I$ be $\ast$-stable and
inclusion-minimal for the $\ast$-twisted conjugacy class of $w_I$. By the
first implication, $\ell^{\ast}(w_J)=r_{\ast}(J)$. Since $w_I$ and $w_J$
are $\ast$-twisted conjugate,
$r_{\ast}(J)=\ell^{\ast}(w_J)=\ell^{\ast}(w_I)=r_{\ast}(I)$.
This implies that $J=I$, and hence $I$ is inclusion-minimal.
This proves \ref{it:class-minimal}.

We already observed that, if $J$ is obtained from $I$ by an elementary
move $J=\omega_K(I)$, then $J$ is spherical and $w_J$ is $\ast$-twisted
conjugate to $w_I$.
Since $K$ is $\ast$-stable, $w_K=w_K^*$, which means that $\omega_K$
preserves $*$-orbits. Hence $\ell^{\ast}(w_J)=\ell^{\ast}(w_I)=
r_{\ast}(I)=r_{\ast}(J)$ and thus $J$ is inclusion-minimal for $\cc$ as well.

Now suppose that $I, J\subseteq S$ are two $\ast$-stable inclusion-minimal subsets
for $\cc$. Since $w_I$ and $w_J$ are $\ast$-twisted conjugate,
we can choose $x\in W$ of minimal length subject to $xw_I=w_Jx^{\ast}$.
Since $\cc\cap W_I=\{w_I\}$, for $s\in I$ we have
$(xs)w_I=xw_Is^{\ast}=w_J(xs)^{\ast}$ and thus
$\ell(xs)>\ell(x)$ for all $s\in I$ by our choice of $x$. Hence $x\in W^I$.
Likewise $\ell(tx)>\ell(x)$ for all $t\in J$ and thus $x\in {}^JW^I$.
This means that $x$ is the unique element of minimal length in $W_JxW_I$.
In the same way, and using $I=I^{\ast}$ and $J=J^{\ast}$, we find that $x^{\ast}$ is the unique
element of minimal length in $W_Jx^{\ast}W_I$. However, $x^{\ast}=w_Jxw_I$ shows that
these double cosets are the same. We therefore have $x=x^{\ast}$ and $xw_I=w_Jx$.

Note that if $K\subseteq S$ is spherical, then $w_K(\Phi_K^+)=\Phi^-_K$ and
$w_K(\Phi^+\setminus\Phi^+_K)=\Phi^+\setminus\Phi^+_K$.
Let $\alpha\in\Phi^+_I$. Then $w_I(\alpha)\in\Phi^-_I$.
Since $x\in W^I$, we have $x(\alpha_s)\in\Phi^+$ for all $s\in I$
and thus $x(\alpha)\in\Phi^+$ and $x(w_I(\alpha))\in\Phi^-$.
This means that $w_J(x(\alpha))\in\Phi^-$ and thus $x(\alpha)\in\Phi^+_J$.
Hence $x(\Phi^+_I)\subseteq\Phi^+_J$. In the same way we find that
$x^{-1}(\Phi^+_J)\subseteq\Phi^+_I$ and thus $x(\Phi^+_I)=\Phi^+_J$.
Hence $x$ maps the simple roots $\Pi_I$ onto the simple roots $\Pi_J$.
We can therefore apply Lemma~\ref{lem:marquis} to obtain a series of
elementary moves which transform $I$ into $J$.
This proves \ref{it:class-moves}.

As a preliminary to proving \ref{it:class-superset}, we assert that
a $\ast$-twisted involution $z$ with $\ell^{\ast}(z)=0$ is $\ast$-twisted
conjugate to $e$.
Indeed, choose $I\subseteq S$ $*$-stable and inclusion-minimal for the twisted
conjugacy class of $z$.
Then by \ref{it:class-minimal} $z$ is $\ast$-twisted conjugate to $w_I$
and $r_{\dm_I}(I)=|I|$.
Hence $0=\ell^{\ast}(z)=\ell^{\ast}(w_I)=r_{\ast}(I)$. This means
that $I=\varnothing$ and thus $w_I=e$.

Now let $I\subseteq J\subseteq S$ be $\ast$-stable and spherical
with $I$ inclusion-minimal for $\cc$.
Suppose first that $w_I$ and $w_J$ are $\ast$-twisted conjugate in $W$.
Since twisted absolute length is constant on twisted conjugacy classes
and compatible with parabolic subgroups, by \ref{it:class-minimal}
and Proposition~\ref{prop:stardiamond} applied to $(W_J,J)$ we have
$r_{\ast}(I)=\ell^{\ast}(w_I)=\ell^{\ast}(w_J)=r_{\ast}(J)+r_{\dm_J}(J)-|J|$.

Conversely, suppose that $r_{\ast}(J)+r_{\dm_J}(J)-|J|=r_{\ast}(I)$.
Again by Proposition~\ref{prop:stardiamond} applied to $(W_J,J)$,
we know that $w_Iw_J$ is a $\dm_J$-twisted involution and that
$\ell^{\dm_J}(w_Iw_J)=r_{\ast}(J)+r_{\dm_J}(J)-|J|-\ell^{\ast}(w_I)=0$.
By the preceding observation, $w_Iw_J$ is therefore
$\dm_J$-twisted conjugate to $e$ in $W_J$. Applying once more
Proposition~\ref{prop:stardiamond}, we find that $w_I$ and $w_J$
are $\ast$-twisted conjugate in $W_J$ and hence also in $W$.
This proves \ref{it:class-superset}.
\end{proof}

The spherical subsets $I\subseteq S$ which are
inclusion-minimal for a conjugacy class in the nontwisted
case coincide with the subsets
satisfying Richardson's \cite{Ric82} $(-1)$-condition.
We will continue to call such subsets \emph{inclusion-minimal},
with the $\ast$-twisted conjugacy class of $w_I$ understood.

In the nontwisted case, for each spherical
$I\subseteq S$ the automorphism $\dm_I$ corresponds
to conjugation by $w_I$, which induces the opposition involution on
$I$. By Proposition~\ref{prop:classif}, the subset $I$ is
inclusion-minimal if and only if the opposition involution on $I$ is
trivial. This is the case if $W_I$ is one of the irreducible
groups $A_1$, $B_n$, $D_{2n}$, $E_7$, $E_8$, $F_4$, $H_3$, $H_4$
and $I_2(2m)$ or a direct product of groups of these types.
Table~\ref{tab:opposition} lists the remaining irreducible spherical
groups $W_I$, the action of their opposition involution, and the
absolute length of their longest element, i.e. the value of
$r_{\ast}(I')=|I'|$ for an inclusion-minimal $I'\subsetneq I$.
The table further indicates the type of the corresponding parabolic subgroup
$W_{I'}$ of $W_I$ (see also \cite{Hum15}). If $W_I$ is a
direct product of spherical irreducible groups, an inclusion-minimal subset
$I'\subseteq I$ with $w_{I'}$ conjugate to $w_I$ is found by considering each
direct factor separately.
\begin{table}[ht]
\centering
\begin{tabular}{cccc}
\hline
$W_I$ & opposition involution & $\ell_T(w_I)$ & $W_{I'}$ \\
\hline
$A_{2n}$ ($n\geq 1$) & $i\mapsto 2n+1-i$ & $n$
& $A_1^n$ \\
$A_{2n-1}$ ($n\geq 2$) & $i\mapsto 2n-i$ & $n$
& $A_1^n$ \\
$D_{2n+1}$ ($n\geq 2$) & $(2n\;\;2n+1)$ & $2n$ & $D_{2n}$ \\
$E_6$  & $(1\;\;6)(3\;\;5)$ & $4$ & $D_4$ \\
$I_2(2m+1)$ ($m\geq 1$) & $(1\;\;2)$ & $1$ & $A_1$ \\
\hline
\end{tabular}
\vspace{4pt}
\caption{The irreducible spherical Coxeter systems with nontrivial
opposition involutions, using Bourbaki vertex numbering.}
\label{tab:opposition}
\end{table}

In the twisted case, a $\ast$-stable subset $I\subseteq S$ is
inclusion-minimal if and only if $\ast$ stabilizes every irreducible component
of $I$ and its restriction to each component coincides with the
opposition involution on that component. In particular, if $\ast$ interchanges
two irreducible components of the graph of $I$, then $I$ is not
inclusion-minimal.

Suppose that $I\subseteq S$ is $\ast$-stable, inclusion-minimal
and admits an elementary move: $J=\omega_K(I)$ with
$K:=I\cup J$, and $T:=J\setminus I$ a single $\ast$-orbit.
Since $I$ is inclusion-minimal, $\ast$ acting on $I$ coincides with
the opposition involution on $I$ and therefore stabilizes each irreducible
component of $I$. Hence
the $\ast$-orbit $\omega_K(T)\subseteq I$ is contained in an
irreducible component of $I$ and therefore also in an
irreducible component $K_0$ of $K$. The opposition involution
on $K_0$ interchanges $T$ and $\omega_{K_0}(T)$ and is
therefore both nontrivial and, since $T$ is a $\ast$-orbit,
different from the permutation
on $K_0$ induced by $\ast$. Inspection of the irreducible
Coxeter graphs whose opposition involution is nontrivial now shows
that $\ast$ is trivial on $K_0$. Hence $T$ consists of a single
element $t=t^{\ast}$, which is not fixed by the opposition involution
on $K_0$. Moreover, $K_0\setminus \{t\}$ is an
inclusion-minimal subset on which $\ast$ acts trivially and
coincides with the opposition involution on $K_0\setminus \{t\}$.
This rules out $A_n$ ($n\ge 3$), $D_{2n+1}$ and $E_6$ as possible
types of $K_0$. The irreducible component $K_0$ is therefore of
type $A_2$ or $I_2(2m+1)$, i.e. its only vertices are $t\notin I$ and
$s=\omega_{K_0}(t)\in I$. Moreover, $s$ and $t$ are both
nonadjacent to $I\setminus\{s\}$, the elementary move swaps
$s$ and $t$, and $m_{st}$ is odd ($K_0=A_2$
corresponding to $m_{st}=3$ and $K_0=I_2(2m+1)$ corresponding to
$m_{st}=2m+1$ for $m\ge 2$).

These considerations prove the forward implication in the following
proposition.
\begin{proposition}
Let $(W,S)$ be a Coxeter system of finite rank defined by a Coxeter
matrix $(m_{st})_{s,t\in S}$ and let $\ast$ be an
involutive automorphism of $W$ preserving $S$.
Then two $\ast$-stable inclusion-minimal subsets $I,J\subseteq S$ are
related by an elementary move precisely when $I\setminus J=\{s\}$ and
$J\setminus I=\{t\}$ with both $s$ and $t$ nonadjacent to $I\cap J$
and $m_{st}$ odd.
\end{proposition}

\begin{proof}
Only the converse remains to be shown. Since $I$ and $J$ are $\ast$-stable,
$\{t\}=J\setminus I$ is also $\ast$-stable and thus a single $\ast$-orbit.
Put $I'=I\cap J$.
Since $m_{st}$ is odd, $W_{\{s,t\}}$ is of type $A_2$ or $I_2(2m+1)$.
Therefore $K=I'\cup\{s,t\}$
is spherical because $I'$ and $\{s,t\}$ are nonadjacent and both spherical.
Moreover,
$\omega_K(I)=\omega_{I'}(I')\cup\omega_{\{s,t\}}(\{s\})=I'\cup\{t\}=J$.
\end{proof}

\begin{example}
We consider the indefinite group $W$ used in the second example
of \cite[Section~5.1]{DeMan99}, with Coxeter graph
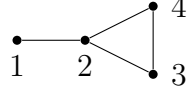
\begin{center}
\begin{tikzpicture}[
    x=0.9cm,
    y=0.9cm,
    every node/.style={
        circle,
        fill=black,
        inner sep=1.2pt
    }
]
\node (1) at (-1,0) {};
\node (2) at (0,0) {};
\node (3) at (1,-0.5) {};
\node (4) at (1, 0.5) {};

\draw (1) -- (2) -- (3) -- (4) -- (2);

\node[fill=none, below=3pt] at (1) {$1$};
\node[fill=none, below=3pt] at (2) {$2$};
\node[fill=none, right=3pt] at (3) {$3$};
\node[fill=none, right=3pt] at (4) {$4$};
\end{tikzpicture}
\end{center}
With $\ast=\id$, there are three conjugacy classes of involutions:
the conjugacy class of the identity element, with $\varnothing$ as the
inclusion-minimal set, one conjugacy class of reflections, with $\{1\}$,
$\{2\}$, $\{3\}$ and $\{4\}$ as its inclusion-minimal sets, and one conjugacy
class of involutions of absolute length 2, with $\{1,3\}$ and $\{1,4\}$ as
its inclusion-minimal sets. We let $\cc$ denote the latter class.
Both $\{1,3\}$ and $\{1,4\}$ are of type $A_1\times A_1$.
The other spherical subsets relevant for $\cc$ are $\{1,2,3\}$ and $\{1,2,4\}$
of type $A_3$ and $\{1,3,4\}$ of type $A_1\times A_2$. All three have two
orbits under their respective opposition involutions, which means that
their longest elements have absolute length $2$ and are in $\cc$.
Writing $[n]_q:=(1-q^n)/(1-q)$, we have $(A_1\times A_1)(q)=[2]_q^2$,
$(A_1\times A_2)(q)=[2]_q^2[3]_q$ and $A_3(q)=[2]_q[3]_q[4]_q$. The lengths
of the longest elements of $A_1\times A_1$, $A_1\times A_2$ and $A_3$ are
$2$, $4$ and $6$. Thus $\cc_{13}(q)=\cc_{14}(q)=q^2$. We further have
\begin{align*}
\cc_{123}(q)+q^{12}\cc_{123}(-q^{-1})&=A_3(q^2)\frac{\cc_{13}(q)}{(A_1\times A_1)(q^2)}
 =\frac{[2]_{q^2}[3]_{q^2}[4]_{q^2}q^2}{[2]_{q^2}^2}\\
&=q^2+q^4+2q^6+q^8+q^{10}.
\end{align*}
and
\begin{align*}
\cc_{134}(q)+q^8\cc_{134}(-q^{-1})&=2(A1\times A_2)(q^2)\frac{\cc_{13}(q)}{(A_1\times A_1)(q^2)}
=2[3]_{q^2}q^2\\
&=2q^2+2q^4+2q^6.
\end{align*}
Hence $\cc_{123}(q)=\cc_{124}(q)=q^2+q^4+q^6$ and $\cc_{134}(q)=2q^2+q^4$. Omitting the
remaining calculations, we further find
$$
W(q)=\frac{(1+q)(1+q^2)(1+q+q^2)}{(1-q)(1-q-q^3)}
$$
and, by applying Corollary~\ref{cor:steinberg},
$$
\frac{\cc(q)}{W(q^2)}=\frac{q^2(1-q^2)(2+q^2+2q^4)}{(1+q^2)(1+q^4)(1+q^2+q^4)}.
$$
Thus
\begin{align*}
\cc(q)&=\frac{2q^2+q^4+2q^6}{1-q^2-q^6}\\
&=2q^2+3q^4+5q^6+7q^8+10q^{10}+15q^{12}+22q^{14}+32q^{16}+\dots.
\end{align*}
\end{example}

\section*{Acknowledgement}
The author acknowledges the use of GPT-5.6 in the investigations
leading to this work and in proofreading the manuscript, for whose
content the author bears sole responsibility.

\end{document}